\documentclass[12pt,reqno]{amsart}
\usepackage{amssymb}
\usepackage{amsmath}
\usepackage{amscd,amsxtra,amsthm}
\usepackage[all]{xy}
\usepackage{etex}
\usepackage{pictex}
\usepackage{graphicx}
\usepackage{mathtools}
\usepackage{float}
\usepackage{xcolor}
\usepackage[utf8]{inputenc}
\usepackage[T1]{fontenc}
\usepackage{subcaption}
\usepackage{tikz}
\usepackage{tkz-tab}
\usepackage{caption}

\usepackage{hyperref}
\hypersetup{colorlinks, citecolor=blue}

\usepackage{mathtools}
\usepackage{multirow}
\DeclarePairedDelimiter{\ceil}{\lceil}{\rceil}

\usepackage{comment}

\newtheorem{lemma}{Lemma}[section]

\newtheorem{theorem}[lemma]{Theorem}

\theoremstyle{conjecture}

\theoremstyle{definition}

\theoremstyle{remark}

\theoremstyle{observation}

\numberwithin{equation}{section}

\begin{document}

\title[3-Color 2-Chromatic Gallai-Ramsey Numbers]{The 3-Color 2-Chromatic Gallai-Ramsey Number for Complete Graphs}

\author{Mark Budden}
\address{Department and Mathematics and Computer Science, Western Carolina University, Cullowhee, NC, USA.}
\email{mrbudden@email.wcu.edu}

\author{Elijah DeJonge}
\address{Department and Mathematics and Computer Science, Western Carolina University, Cullowhee, NC, USA.}
\email{dejongee@email.wcu.edu}



\begin{abstract}
For $1\le i\le j\le k$, let $gr^3_2(K_i, K_j, K_k)$ be the least $n\in \mathbb{N}$ such that every rainbow triangle-free $3$-coloring of $K_n$ (using red, blue, and green) contains a red/blue $K_i$, a red/green $K_j$, or a blue/green $K_k$.  In this paper, we completely evaluate $gr^3_2(K_i, K_j, K_k)$ for all values of $i$, $j$, and $k$ in which $r(K_i,K_j)$ is known.
\end{abstract}

\maketitle

\section{Introduction}\label{intro}

In 1977, K. Chung, M. Chung, and C. Liu \cite{CCL} proposed an interesting variation of multicolor Ramsey numbers.  Rather than finding the threshold at which a monochromatic subgraph exists within a $t$-colored complete graph, they sought the threshold at which a subgraph exists whose edges use at most $s$ colors, where $1\le s<t$.  The chromatic Ramsey numbers they defined were some of the first ``weakened'' Ramsey numbers to be studied (cf. \cite{B2, HM}).  In this paper, we consider chromatic Ramsey numbers for Gallai colorings (those that lack rainbow triangles).  Before describing our results, we first review the necessary definitions and background.

A {\it $t$-coloring} of a graph $G$ is a map $f:E(G)\longrightarrow \{1, 2, \dots , t\}$.  Such maps are not assumed to be surjective (i.e., a $t$-coloring may use fewer than $t$ colors).   If a $t$-coloring is surjective, then it is called an {\it exact $t$-coloring}.   If $t\ge 2$ and $G_1, G_2, \dots , G_t$ are graphs, then the {\it $t$-color Ramsey number} $r(G_1, G_2, \dots , G_t)$ is the least $n\in \mathbb{N}$ such that every $t$-coloring of the edges of $K_n$ results in a subgraph isomorphic to $G_i$ in color $i$, for some $i\in \{ 1, 2, \dots , t\}$.  A {\it critical coloring} for $r(G_1, G_2, \dots, G_t)$ is a $t$-coloring of $K_{r(G_1, G_2, \dots , G_t)-1}$ that avoids a copy of $G_i$ in color $i$, for all $i\in \{1, 2, \dots , t\}$.  A detailed overview of known values of $r(G_1, G_2, \dots , G_t)$ can be found in Radziszowski's dynamic survey \cite{Rad}.

Let $S_i:=\{1, 2, \dots , t\}\setminus \{t+1-i\}$.  Then the {\it $t$-color $(t-1)$-chromatic Ramsey number} $r^t_{t-1}(G_1, G_2, \dots , G_t)$ is defined to be the least $n\in \mathbb{N}$ such that every $t$-coloring of $K_n$ contains a copy of $G_i$ spanned by edges whose colors are in the set $S_i$, for some $i\in \{1, 2, \dots , t\}$.  A {\it critical coloring} for $r^t_{t-1}(G_1, G_2, \dots , G_t)$ is a $t$-coloring of $K_{r^t_{t-1}(G_1, G_2, \dots , G_t)-1}$ that lacks a $G_i$ whose edges use only the colors in $S_i$, for every $i\in \{1, 2, \dots , t\}$.  A more general class of chromatic Ramsey numbers than those defined here have also been studied in \cite{CL, J1, J2, LC, YL}.  In \cite{CL}, the following $3$-color $2$-chromatic Ramsey numbers were determined:
$$r^3_2(K_3, K_3, K_k)=\begin{cases} 5 & \mbox{if $3\le k\le 5$} \\ 6 & \mbox{if $k\ge 6$,}\end{cases}$$ 
$$r^3_2(K_3, K_4, K_k)=\begin{cases} 7 & \mbox{if $k=4$} \\ 8 & \mbox{if $5\le k\le 6$} \\ 9 & \mbox{if $k\ge 7$,}\end{cases}$$
$$r^3_2(K_3, K_5, K_k)=\begin{cases} 9 & \mbox{if $k=5$} \\ 11 & \mbox{if $k=6$} \\ 12 & \mbox{if $7\le k\le 8$} \\ 13 & \mbox{if $9\le k\le 13$} \\ 14 & \mbox{if $k\ge 14$,}\end{cases}$$
and $r^3_2(K_4, K_4, K_4)=10$.

A {\it Gallai $t$-coloring} of a graph $G$ is a $t$-coloring of $G$ that avoids rainbow triangles.  That is, every $K_3$-subgraph of $G$ has at most two colors on its edges.
The {\it $t$-color $(t-1)$-chromatic Gallai-Ramsey number} $gr^t_{t-1}(G_1, G_2, \dots , G_t)$ is the least $n\in \mathbb{N}$ such that every Gallai $t$-coloring of $K_n$ contains a copy of $G_i$ spanned by edges whose colors are in the set $S_i$, for some $i\in \{1, 2, \dots , t\}$.
Since every Gallai $t$-coloring of $K_n$ is a $t$-coloring of $K_n$, it follows that $$gr^t_{t-1}(G_1, G_2,\dots ,G_t)\le r^t_{t-1}(G_1, G_2, \dots , G_t).$$  A {\it critical coloring} for $gr^t_{t-1}(G_1, G_2, \dots , G_t)$ is a Gallai $t$-coloring of $K_{gr^t_{t-1}(G_1, G_2, \dots , G_t)-1}$ that avoids a copy of $G_i$ whose edges use only the colors in $S_i$, for all $i\in \{1, 2, \dots , t\}$.

Our focus in this paper will be on the determination of chromatic Gallai-Ramsey numbers of the form $gr^3_2(K_i, K_j, K_k)$, where we assume that $i\le j\le k$, unless it is explicitly stated otherwise.  We will refer to the colors $1$, $2$, and $3$ as red, blue, and green, respectively.
One benefit of working with Gallai $t$-colorings of $K_n$ (as opposed to $t$-colorings of $K_n$) is that the proofs of upper bounds for chromatic Gallai-Ramsey numbers can be broken into manageable cases by using the following structural theorem, which is a reinterpretation of a result of Gallai \cite{Gallai}. 

\begin{theorem}[\cite{GS}]\label{Gallaistruct}
Every Gallai-colored complete graph can be formed by replacing the vertices in a $2$-colored complete graph of order at least $2$ with Gallai-colored complete graphs.
\end{theorem}

In Theorem \ref{Gallaistruct}, the $2$-colored complete graph with order at least $2$ is called the {\it base graph}, and the Gallai-colored complete graphs that replace the vertices in the base graph are called the {\it blocks}.
The following lemma will save us some time when using Theorem \ref{Gallaistruct}.

\begin{lemma}[\cite{MSn}]\label{not3}
If $\mathcal{B}$ is the base graph of a Gallai partition, chosen to have minimal order, then $|V(\mathcal{B})|\ne 3$.
\end{lemma}

\noindent We will also need the following lemma that was proved by Erd\H{o}s, Simonovits, and S\'os \cite{ESS} in 1975.

\begin{lemma}[\cite{ESS}]\label{numcolors}
Every Gallai coloring of $K_n$, where $n\ge 2$, uses at most $n-1$ colors.
\end{lemma}

\noindent  The next lemma can be compared to Lemma 5 of \cite{BB}.  It follows from the observation that the addition of any vertex to a Gallai $s$-colored $K_m$-subgraph introduces at most one new color to the resulting edge set.

\begin{lemma}\label{increaseorder}
If a Gallai $t$-coloring of $K_n$ contains a $K_m$-subgraph whose edges use at most $s<t-1$ colors, where $m<n$, then it contains a $K_{m+1}$-subgraph whose edges use at most $s+1$ colors.
\end{lemma}

In Section \ref{main1}, several general results and bounds are decribed that will be needed in the computation of $gr^3_2(K_i, K_j, K_k)$.  We will find that the complete evaluation of $gr^3_2(K_i, K_j, K_k)$ is possible when $r(K_i,K_j)$ is known.    The known nontrivial classical $2$-color Ramsey numbers are $r(K_3,K_3)=6$ \cite{GG}, $r(K_3, K_4)=9$ \cite{GG}, $r(K_3, K_5)=14$ \cite{GG}, $r(K_3, K_6)=18$ \cite{K}, $r(K_3, K_7)=23$ \cite{GY, Kalb}, $r(K_3, K_8)=28$ \cite{GR, MZ}, $r(K_3, K_9)=36$ \cite{GR, Kalb}, $r(K_4, K_4)=18$ \cite{GG}, and $r(K_4, K_5)=25$ \cite{Ka1, MR4}.  In Section \ref{main2}, we focus on chromatic Gallai-Ramsey numbers of the form $gr^3_2(K_3, K_j, K_k)$, when $j\in \{3, 4, 5, 6, 7, 8, 9\}$ and $k\ge j$.  In  Section \ref{main3}, we determine the values of $gr^3_2(K_4, K_j, K_k)$, when $j\in\{4,5\}$ and $k\ge j$.  Finally, we conclude by describing some directions for future work on this topic.

\section{General Results for $gr^3_2(K_i, K_j, K_k)$}\label{main1}



When $1\le j\le k$, every complete graph contains at least one vertex, and every $K_1$ contains a red/blue $K_1$ (the colors are irrelevant since $K_1$ does not contain any edges).  It follows that $$gr^3_2(K_1, K_j, K_k)=1, \qquad \mbox{for all $1\le j\le k$.}$$
If $2\le j\le k$, then $gr^3_2(K_2, K_j, K_k)\ge j$ follows from considering a green $K_{j-1}$, which lacks a red/blue $K_2$, a red/green $K_{j}$, and a blue/green $K_k$.  The inequality $gr^3_2(K_2, K_j, K_k)\le j$ follows from the fact that every Gallai $3$-coloring of $K_{j}$ contains a red/blue $K_2$ (that is, a red edge or a blue edge), or all edges are green, forming a red/green $K_j$.  Thus, $$gr^3_2(K_2, K_j, K_k)=j,\qquad \mbox{for all $2
\le j\le k$.}$$

The following theorem provides a general upper bound for the chromatic Gallai-Ramsey number $gr^3_2(K_i, K_j, K_k)$, with equality when $k$ is sufficiently large (cf., Theorem 2.2 of \cite{KO}).

\begin{theorem}\label{upperbound}
For all $3\le i\le j\le k$, $$gr^3_2(K_i, K_j, K_k)\le r(K_i, K_j).$$
If $k\ge r(K_i, K_j)$, then $$gr^3_2(K_i, K_j, K_k)=r(K_i, K_j).$$
\end{theorem}

\begin{proof}
Let $p=r(K_i, K_j)$ and consider a Gallai $3$-coloring of $K_p$.  Grouping together the colors red and blue, it follows that there exists a red/blue $K_i$ or a green $K_j$ (and hence, a red/green $K_j$).  Therefore, $gr^3_2(K_i, K_j, K_k)\le p$.  

Now assume that $k\ge r(K_i, K_j)$ and let $q=gr^3_2(K_i, K_j, K_k)$.  Consider a blue/green coloring of $K_q$.  Since it can be viewed as a Gallai $3$-coloring of $K_q$, it follows that there exists a blue $K_i$ (and hence, a red/blue $K_i$) or a green $K_j$ (and hence, a red/green $K_j$).  It follows that $r(K_i, K_j)\le q$.
\end{proof}

This theorem also sheds some light on determining which values of $gr^3_2(K_i, K_j, K_k)$ we can evaluate (i.e., those in which $r(K_i, K_j)$ is known). 
The next theorem limits the rate at which $gr^3_2(K_i, K_j, K_k)$ grows, relative to $k$.  Unlike elsewhere in this paper, it is not necessary to assume an ordering on $i$, $j$, and $k$ in this theorem.

\begin{theorem}\label{rededge}
If $i,j,k\ge 3$ and $gr^3_2(K_i,K_j, K_{k-1})\le p$, then in every Gallai $3$-coloring of $K_{p+1}$, every vertex is incident with a red edge, or there exists a red/blue $K_i$, a red/green $K_j$, or a blue/green $K_k$.  
\end{theorem}

\begin{proof}
Let $gr^3_2(K_i,K_j, K_{k-1})\le p$ and consider a Gallai $3$-coloring of $K_{p+1}$.  If there exists a vertex $x$ that is only incident with edges that are blue and green, then deleting $x$ results in a Gallai $3$-coloring of $K_p$.  This coloring necessarily contains a red/blue $K_i$, a red/green $K_j$, or a blue/green $K_{k-1}$.  In the last case, including vertex $x$ results in a blue/green $K_k$.  If no such $x$ exists, then every vertex is incident with a red edge.
\end{proof}



\section{The Chromatic Gallai-Ramsey Number $gr^3_2(K_3, K_j, K_k)$}\label{main2}

We begin this section with two lemmas that provide useful upper and lower bounds for $gr^3_2(K_3, K_j, K_k)$.

\begin{lemma}\label{GallaiKn}
For $\ell\in \mathbb{N}$, let $j\ge 2\ell +2$ and $$n\ge r(K_3, K_{j-2\ell})+2\ell.$$  Then every exact Gallai $3$-coloring of $K_n$ that contains a red matching of size $\ell$ must contain a red/blue $K_3$ or a red/green $K_j$.
\end{lemma} 

\begin{proof}
Consider a Gallai $3$-coloring of $K_n$ with a red matching of size $\ell$.  Suppose that $x_1x_2, x_3x_4, \dots , x_{2\ell-1}x_{2\ell}$ are pairwise disjoint red edges and denote the other vertices by $y_1, y_2, \dots, y_{n-2\ell}$.  If an edge $x_iy_j$ is blue, then either the subgraph induced by $\{x_{i-1},x_i,y_j\}$ or the subgraph induced by $\{x_i, x_{i+1},y_j\}$ is a red/blue $K_3$.  So, assume that all edges joining $X:=\{x_1, x_2, \dots , x_{2\ell}\}$ to $Y:=\{y_1, y_2, \dots , y_{n-2\ell}\}$ are red or green.  Since $n-2\ell \ge r(K_3, K_{j-2\ell})$, the subgraph induced by $Y$ contains a blue $K_3$ (and hence, a red/blue $K_3$) or a red/green $K_{j-2\ell}$.  In the latter case, including the vertices in $X$ results in a red/green $K_j$.
\end{proof}

\begin{lemma}\label{GallaiED}
Let $k\ge j\ge 3$ and assume that $0\leq  \ell < \frac{j}{2}-1$ and $$ k\le r(K_3, K_{j-n})+\ell.$$  If $G$ is a triangle-free graph with independence number $\ell$ and order $n$, then 
$$gr^3_2(K_3, K_j, K_k) \geq k-\ell+n.$$ 
\end{lemma}

\begin{proof}
Suppose that $0\leq  \ell < \frac{j}{2}-1$ and let $G$ be a triangle-free graph with independence number $\ell$ and order $n$.  Let $\mathcal{G}$ be a red/green coloring of $K_n$ in which the red subgraph is isomorphic to $G$ and the green subgraph is isomorphic to $\overline{G}$.  
Since $k\le r(K_3, K_{j-n})+\ell$, there exists a blue/green $K_{k-\ell-1}$, which we denote by $\mathcal{H}$,  that avoids a blue $K_3$ and a green $K_{j-n}$.   Take the disjoint union of $\mathcal{G}$ and $\mathcal{H}$ and join them together with green edges.  By construction, no red/blue $K_3$ exists. The largest red/green complete graph that can be formed uses all of the vertices in $\mathcal{G}$ and the vertices from a complete green subgraph in $\mathcal{H}$.  Hence, it will have order at most $n+j-n-1=j-1$.   The largest blue/green complete graph that can be formed uses all of the vertices in $\mathcal{H}$ and $\ell$ independent vertices in the red subgraph of $\mathcal{G}$.  Hence, it will have order at most $k-\ell-1+\ell=k-1$.  It follows that $gr^3_2(K_3, K_j, K_k)\ge k-\ell+n$.
\end{proof}

Now we are ready to focus on the evaluation of $gr^3_2(K_3, K_j, K_k)$, for the cases where $r(K_3, K_j)$ is known. 

\begin{theorem}
If $k\ge 3$, then $$gr^3_2(K_3, K_3, K_k)=\begin{cases} k & \mbox{if $3\le k\le 5$} \\ 6 & \mbox{if $k\ge 6$.} \end{cases}$$
\end{theorem}

\begin{proof} Since $r(K_3, K_3)=6$ \cite{GG}, Theorem \ref{upperbound} implies that $$gr^3_2(K_3, K_3, K_k)=6 \quad \mbox{if $k\ge 6$.}$$  So, assume that $3\le k\le 5$.  Since $r(K_3, K_3)=6$, there exists a blue/green $K_{k-1}$ that avoids a blue $K_3$ (and hence, a red/blue $K_3$) and a green $K_3$ (and hence, a red/green $K_3$).   It follows that $gr^3_2(K_3, K_3, K_k)\ge k$.  To prove the reverse inequality, consider a Gallai $3$-coloring of $K_k$ that avoids a blue/green $K_k$ (so, there exists at least one red edge).   Since $$k\ge r(K_3, K_1)+2=3,$$ Lemma \ref{GallaiKn} implies that there exists a red/blue $K_3$, a red/green $K_3$, or a blue/green $K_k$.  It follows that $gr^3_2(K_3, K_3, K_k)\le k$.
\end{proof}

\begin{theorem}\label{gr34k}
If  $k\ge 4$, then $$gr^3_2(K_3, K_4, K_k)=\begin{cases} 5 & \mbox{if $k=4$} \\ k & \mbox{if $5\le k\le 8$} \\ 9 & \mbox{if $k\ge 9$.}\end{cases}$$
\end{theorem}

\begin{proof} Since $r(K_3, K_4)=9$ \cite{GG}, Theorem \ref{upperbound} implies that $$gr^3_2(K_3, K_4, K_k)=9 \quad \mbox{if $k\ge 9$.}$$ For $k<9$, we separate the proof into the following cases.

\underline{Case 1:} Assume that $k=4$.  If we let $G\cong K_2$ (so that $\ell=1$ and $n=2$) in Lemma \ref{GallaiED} and note that $k\le r(K_3, K_2)+\ell =4$, then it follows that $gr^3_2(K_3, K_4, K_4)\ge 5$.
To prove the reverse inequality, consider a Gallai $3$-coloring of $K_{5}$ that avoids a blue/green $K_4$.  Then there exists a red matching of size at least two (otherwise, it is possible to delete a single vertex that removes all red edges, leaving a blue/green $K_5$).  Since $$5\ge r(K_3, K_1)+4=5,$$ Lemma \ref{GallaiKn} implies that there exists a red/blue $K_3$, a red/green $K_4$, or a blue/green $K_4$.  It follows that $gr^3_2(K_3, K_4, K_4)\le 5$.

\underline{Case 2:} Assume that $5\le k\le 8$.  Since $r(K_3, K_4)=9$, there exists a blue/green $K_{k-1}$ that avoids a blue $K_3$ (and hence, a red/blue $K_3$) and a green $K_4$ (and hence, a red/green $K_4$).   It follows that $gr^3_2(K_3, K_4, K_k)\ge k$.  To prove the reverse inequality, consider a Gallai $3$-coloring of $K_k$ that avoids a blue/green $K_k$ (so, there exists at least one red edge).   Since $$k\ge r(K_3, K_2)+2=5,$$ Lemma \ref{GallaiKn} implies that there exists a red/blue $K_3$, a red/green $K_4$, or a blue/green $K_k$.  It follows that $gr^3_2(K_3, K_4, K_k)\le k$.
\end{proof}

\begin{theorem}
If $k\ge 5$, then $$gr^3_2(K_3, K_5, K_k)=\begin{cases} k+1 & \mbox{if $5\le k\le 7$} \\ k & \mbox{if $8\le k\le 13$} \\ 14 & \mbox{if $k\ge 14$.}\end{cases}$$
\end{theorem}

\begin{proof} Since $r(K_3, K_5)=14$ \cite{GG}, Theorem \ref{upperbound} implies that $$gr^3_2(K_3, K_5, K_k)=14 \quad \mbox{if $k\ge 14$.}$$ For $k<14$, we separate the proof into the following cases.

\underline{Case 1:} Assume that $5\le k\le 7$.  If we let $G\cong K_2$ (so that $\ell=1$ and $n=2$) in Lemma \ref{GallaiED} and note that $k\le r(K_3, K_3)+\ell =7$, then it follows that $gr^3_2(K_3, K_5, K_k)\ge k+1$.
To prove the reverse inequality, consider a Gallai $3$-coloring of $K_{k+1}$ that avoids a blue/green $K_k$.  Then there exists a red matching of size at least two (otherwise, it is possible to delete a single vertex that removes all red edges, leaving a blue/green $K_k$).  Since $k+1\ge r(K_3, K_1)+4=5,$ Lemma \ref{GallaiKn} implies that there exists a red/blue $K_3$, a red/green $K_5$, or a blue/green $K_k$.  It follows that $gr^3_2(K_3, K_5, K_k)\le k+1$.

\underline{Case 2:} Assume that $8\le k\le 13$.  Since $r(K_3, K_5)=14$, there exists a blue/green $K_{k-1}$ that avoids a blue $K_3$ (and hence, a red/blue $K_3$) and a green $K_5$ (and hence, a red/green $K_5$).   It follows that $gr^3_2(K_3, K_5, K_k)\ge k$.  To prove the reverse inequality, consider a Gallai $3$-coloring of $K_k$ that avoids a blue/green $K_k$ (so, there exists at least one red edge).   Since $$k\ge r(K_3, K_3)+2=8,$$ Lemma \ref{GallaiKn} implies that there exists a red/blue $K_3$, a red/green $K_5$, or a blue/green $K_k$.  It follows that $gr^3_2(K_3, K_5, K_k)\le k$.
\end{proof}

\begin{theorem}
If $k\ge 6$, then $$gr^3_2(K_3, K_6, K_k)=\begin{cases} k+1 & \mbox{if $6\le k \le 10$} \\ k & \mbox{if $11\le k \le 17$} \\ 18 & \mbox{if $k\ge 18$.}\end{cases}$$
\end{theorem}

\begin{proof} Since $r(K_3, K_6)=18$ \cite{K}, Theorem \ref{upperbound} implies that $$gr^3_2(K_3, K_6, K_k)=18 \quad \mbox{if $k\ge 18$.}$$ For $k<18$, we separate the proof into the following cases.

\underline{Case 1:} Assume that $6\le k\le 10$.   If we let $G\cong K_2$ (so that $\ell=1$ and $n=2$) in Lemma \ref{GallaiED} and note that $k\le r(K_3, K_4)+\ell =10$, then it follows that $gr^3_2(K_3, K_6, K_k)\ge k+1$.
To prove the reverse inequality, consider a Gallai $3$-coloring of $K_{k+1}$ that avoids a blue/green $K_k$.  Then there exists a red matching of size at least two (otherwise, it is possible to delete a single vertex that removes all red edges, leaving a blue/green $K_k$).  Since $k+1\ge r(K_3, K_2)+4=7,$ Lemma \ref{GallaiKn} implies that there exists a red/blue $K_3$, a red/green $K_6$, or a blue/green $K_k$.  It follows that $gr^3_2(K_3, K_6, K_k)\le k+1$.

\underline{Case 2:} Assume that $11\le k\le 17$.  Since $r(K_3, K_6)=18$, there exists a blue/green $K_{k-1}$ that avoids a blue $K_3$ (and hence, a red/blue $K_3$) and a green $K_6$ (and hence, a red/green $K_6$).   It follows that $gr^3_2(K_3, K_6, K_k)\ge k$.  To prove the reverse inequality, consider a Gallai $3$-coloring of $K_k$ that avoids a blue/green $K_k$ (so, there exists at least one red edge).   Since $$k\ge r(K_3, K_4)+2=11,$$ Lemma \ref{GallaiKn} implies that there exists a red/blue $K_3$, a red/green $K_6$, or a blue/green $K_k$.  It follows that $gr^3_2(K_3, K_6, K_k)\le k$.
\end{proof}

\begin{theorem}
If $k\ge 7$, then $$gr^3_2(K_3, K_7, K_k)=\begin{cases} k+2 & \mbox{if $7\le k\le 8$} \\ k+1 & \mbox{if $9\le k\le 15$} \\ k & \mbox{if $16\le k \le 22$} \\ 23 & \mbox{if $k\ge 23$.} \end{cases}$$
\end{theorem}

\begin{proof} Since $r(K_3, K_7)=23$ \cite{GY, Kalb}, Theorem \ref{upperbound} implies that $$gr^3_2(K_3, K_7, K_k)=23 \quad \mbox{if $k\ge 23$.}$$ For $k<23$, we separate the proof into the following cases.

\underline{Case 1:}  Assume that $7\le k\le 8$.  If we let $G\cong 2K_2$ (so that $\ell =2$ and $n=4$) in Lemma \ref{GallaiED} and note that $k\le r(K_3, K_3)+\ell =8$, then it follows that $gr^3_2(K_3, K_7, K_k)\ge k+2$.  To prove the reverse inequality, consider a Gallai $3$-coloring of $K_{k+2}$ that avoids a red/blue $K_3$, a red/green $K_7$, and a blue/green $K_k$.  Let $X$ be the set of vertices that are incident with red edges and let $Y$ be the set of vertices that are not in $X$.  Note that all edges joining vertices in $X$ to vertices in $Y$ are green, all red edges are contained in the subgraph induced by $X$, and all blue edges are contained in the subgraph induced by $Y$.

Since the subgraph induced by $X\cup \{y\}$, for any $y\in Y$, is a complete red/green subgraph, assume that $|X|\le 5$.  Since the subgraph induced by $Y\cup \{x\}$, for any $x\in X$, is a complete blue/green subgraph, assume that $|Y|\le k-2$, which implies that $|X|\ge k+2-(k-2)=4$.  The subgraph induced by $X$ must contain at least one green edge (as a red $K_3$ is avoided), which we denote by $x_1x_2$.  So, the subgraph induced by $Y\cup \{x_1,x_2\}$ is a blue/green complete subgraph, implying that $|Y|\le k-3$.  Thus, $|X|=5$ and $|Y|=k-3$.  From the assumption that $k\ge 7$, we find that $|Y|\ge 4$, from which it follows that $Y$ contains a green edge (since a blue $K_3$ is avoided).  Denoting such a green edge by $y_1y_2$, the subgraph induced by $X\cup \{y_1, y_2\}$ must be a red/green $K_7$, leading to a contradiction.  Hence, every Gallai $3$-coloring of $K_{k+2}$ contains a red/blue $K_3$, a red/green $K_7$, or a blue/green $K_k$, when $7\le k\le 8$.  It follows that $gr^3_2(K_3, K_7, K_k)\le k+2$.

\underline{Case 2:} Assume that $9\le k\le 15$.   If we let $G\cong K_2$ (so that $\ell=1$ and $n=2$) in Lemma \ref{GallaiED} and note that $k\le r(K_3, K_5)+\ell =15$, then it follows that $gr^3_2(K_3, K_7, K_k)\ge k+1$.
To prove the reverse inequality, consider a Gallai $3$-coloring of $K_{k+1}$ that avoids a blue/green $K_k$.  Then there exists a red matching of size at least two (otherwise, it is possible to delete a single vertex that removes all red edges, leaving a blue/green $K_k$).  Since $k+1\ge r(K_3, K_3)+4=10,$ Lemma \ref{GallaiKn} implies that there exists a red/blue $K_3$, a red/green $K_7$, or a blue/green $K_k$.  It follows that $gr^3_2(K_3, K_7, K_k)\le k+1$.

\underline{Case 3:} Assume that $16\le k\le 22$.  Since $r(K_3, K_7)=23$, there exists a blue/green $K_{k-1}$ that avoids a blue $K_3$ (and hence, a red/blue $K_3$) and a green $K_7$ (and hence, a red/green $K_7$).   It follows that $gr^3_2(K_3, K_7, K_k)\ge k$.  To prove the reverse inequality, consider a Gallai $3$-coloring of $K_k$ that avoids a blue/green $K_k$ (so, there exists at least one red edge).   Since $$k\ge r(K_3, K_5)+2=16,$$ Lemma \ref{GallaiKn} implies that there exists a red/blue $K_3$, a red/green $K_7$, or a blue/green $K_k$.  It follows that $gr^3_2(K_3, K_7, K_k)\le k$.
\end{proof}

\begin{theorem}
If $k\ge 8$, then $$gr^3_2(K_3, K_8, K_k)=\begin{cases} 11 & \mbox{if $8\le k\le 9$} \\ k+2 & \mbox{if $10\le k\le 11$} \\ k+1 & \mbox{if $12\le k\le 19$} \\ k & \mbox{if $20\le k\le 27$} \\ 28 & \mbox{if $k\ge 28$.} \end{cases}$$
\end{theorem}

\begin{proof} Since $r(K_3, K_8)=28$ \cite{GR, MZ}, Theorem \ref{upperbound} implies that $$gr^3_2(K_3, K_8, K_k)=28 \quad \mbox{if $k\ge 28$.}$$ For $k<28$, we separate the proof into the following cases.

\underline{Case 1:} Assume that $8\le k\le 9$.  If we let $G\cong C_5$ (so that $\ell =2$ and $n=5$) in Lemma \ref{GallaiED} and note that $8\le r(K_3,K_3)+\ell=8$, then it follows that $$gr^3_2(K_3, K_8, K_9)\ge gr^3_2(K_3, K_8, K_8)\ge 11.$$
Next, consider a Gallai $3$-coloring of $K_{11}$ that avoids a red/blue $K_3$, a red/green $K_8$, and a blue/green $K_9$.  Let $X$ be the set of vertices that are incident with red edges and let $Y$ be the vertices that are not incident with red edges.  Note that all edges joining vertices in $X$ to vertices in $Y$ are green, all red edges are contained in the subgraph induced by $X$, and all blue edges are contained in the subgraph induced by $Y$.  

Since the subgraph induced by $X\cup \{y\}$, for any $y\in Y$, is a complete red/green subgraph, assume that $|X|\le 6$.  Since the subgraph induced by $Y\cup \{x\}$, for any $x\in X$, is a complete blue/green subgraph, for any $x\in X$, assume that $|Y|\le 7$.  It follows that $4\le |X|\le 6$.  If $|X|=4$ and $|Y|=7$, note that $X$ must contain at least one green edge.  The endpoints of such an edge, along with all of the vertices in $Y$, induce a blue/green $K_9$.  If $|X|=5$ and $|Y|=6$, then $r(K_3, K_3)=6$ implies that $Y$ must contain a green $K_3$.  The vertices in this $K_3$, along with all of the vertices in $X$, induce a red/green $K_8$.  If $|X|=6$ and $|Y|=5$, then $Y$ must contain at least one green edge.  The endpoints of such an edge, along with all of the vertices in $X$, induce a red/green $K_8$.  It follows that every Gallai $3$-coloring of $K_{11}$ contains a red/blue $K_3$, a red/green $K_8$, or a blue/green $K_9$.  Hence, $gr^3_2(K_3, K_8, K_9)\le 11,$ from which it follows that $$gr^3_2(K_3, K_8, K_8)=gr^3_2(K_3, K_8, K_9)=11.$$

\underline{Case 2:} Assume that $10\le k\le 11$.  If we let $G\cong 2K_2$ (so that $\ell=2$ and $n=4$) in Lemma \ref{GallaiED} and note that $k\le r(K_3, K_4)+\ell =11$, then it follows that $gr^3_2(K_3, K_8, K_k)\ge k+2$.
To prove the reverse inequality, consider a Gallai $3$-coloring of $K_{k+2}$ that avoids a red/blue $K_3$, a red/green $K_8$, and a blue/green $K_k$.  Let $X$ be the set of vertices that are incident with red edges, and let $Y$ be the set of vertices that are not incident with red edges.

Since the subgraph induced by $Y\cup \{x\}$, for any $x\in X$, is a complete blue/green subgraph, assume that $|Y|\le k-2$.  It follows that $|X|\ge 4$, so $X$ must contain at least one green edge, which we denote by $x_1x_2$.  The subgraph induced by $Y\cup \{x_1,x_2\}$ is a blue/green complete subgraph, from which it follows that $|Y|\le k-3$ and $|X|\ge 5$.   Now, $|Y|\ge k-3\ge 7$ and $r(K_3, K_3)=6$ \cite{GG} imply that the subgraph induced by $Y$ contains a green $K_3$.  This $K_3$, along with five of the vertices in $X$, form a red/green $K_8$.  It follows that $gr^3_2(K_3, K_8, K_k)\le k+2$.

\underline{Case 3:} Assume that $12\le k\le 19$.  If we let $G\cong K_2$ (so that $\ell=1$ and $n=2$) in Lemma \ref{GallaiED} and note that $k\le r(K_3, K_6)+\ell =19$, then it follows that $gr^3_2(K_3, K_8, K_k)\ge k+1$.
To prove the reverse inequality, consider a Gallai $3$-coloring of $K_{k+1}$ that avoids a blue/green $K_k$.  Then there exists a red matching of size at least two (otherwise, it is possible to delete a single vertex that removes all red edges, leaving a blue/green $K_k$).  Since $k+1\ge r(K_3, K_4)+4=13,$ Lemma \ref{GallaiKn} implies that there exists a red/blue $K_3$, a red/green $K_8$, or a blue/green $K_k$.  It follows that $gr^3_2(K_3, K_8, K_k)\le k+1$.

\underline{Case 4:} Assume that $20\le k\le 27$.  Since $r(K_3, K_8)=28$, there exists a blue/green $K_{k-1}$ that avoids a blue $K_3$ (and hence, a red/blue $K_3$) and a green $K_8$ (and hence, a red/green $K_8$).   It follows that $gr^3_2(K_3, K_8, K_k)\ge k$.  To prove the reverse inequality, consider a Gallai $3$-coloring of $K_k$ that avoids a blue/green $K_k$ (so, there exists at least one red edge).   Since $$k\ge r(K_3, K_6)+2=20,$$ Lemma \ref{GallaiKn} implies that there exists a red/blue $K_3$, a red/green $K_8$, or a blue/green $K_k$.  It follows that $gr^3_2(K_3, K_8, K_k)\le k$.
\end{proof}

\begin{theorem}
If $k\ge 9$, then $$gr^3_2(K_3, K_9, K_k)=\begin{cases} k+3 & \mbox{if $9\le k\le 11$} \\ k+2 & \mbox{if $12\le k \le 16$} \\ k+1 & \mbox{if $17\le k \le 24$} \\ k & \mbox{if $25\le k \le 35$} \\ 36 & \mbox{if $k\ge 36$.} \end{cases}$$
\end{theorem}

\begin{proof} Since $r(K_3, K_9)=36$ \cite{GR, Kalb}, Theorem \ref{upperbound} implies that $$gr^3_2(K_3, K_9, K_k)=36 \quad \mbox{if $k\ge 36$.}$$ For $k<36$, we separate the proof into the following cases.

\underline{Case 1:} Suppose that $9\le k\le 11$.  If we let $G\cong C_5$ (so that $\ell =2$ and $n=5$) in Lemma \ref{GallaiED} and note that $k\le r(K_3, K_4)+\ell =11$, then it follows that $gr^3_2(K_3, K_9, K_k)\ge k+3$.  To prove the reverse inequality, consider a Gallai $3$-coloring of $K_{k+3}$ that avoids a red/blue $K_3$, a red/green $K_9$, and a blue/green $K_k$.  Let $X$ be the set of vertices that are incident with red edges and let $Y$ be the vertices that are not incident with red edges.  Note that all edges joining vertices in $X$ to vertices in $Y$ are green, all red edges are contained in the subgraph induced by $X$, and all blue edges are contained in the subgraph induced by $Y$.  

Since the subgraph induced by $Y\cup \{x\}$, for any $x\in X$, is a complete blue/green subgraph, assume that $|Y|\le k-2$ and $|X|\ge k+3-(k-2)=5$.  So, $X$ must contain at least one green edge, which we denote by $x_1x_2$.  The subgraph induced by $Y\cup \{x_1,x_2\}$ is a blue/green complete subgraph, from which it follows that $|Y|\le k-3$ and $|X|\ge 6$.  Since the subgraph induced by $X\cup \{y\}$, for any $y\in Y$, is a red/green complete subgraph, assume that $|X|\le 7$ and $|Y|\ge k+3-7=k-4$.  

If $|X|=6$ and $|Y|=k-3\ge 6$, then $r(K_3, K_3)=6$ \cite{GG} implies that $Y$ contains a green $K_3$.  This green $K_3$, along with $X$, forms a red/green $K_9$.  If $|X|=7$ and $|Y|=k-4\ge 5$, then $Y$ must contain a green edge, which we denote by $y_1y_2$.  The subgrpah induced by $X\cup \{y_1, y_2\}$ is a red/green $K_9$.  It follows that $gr^3_2(K_3, K_9, K_k)\le k+3$. 

\underline{Case 2:} Suppose that $12\le k\le 16$.  If we let $G\cong 2K_2$ (so that $\ell=2$ and $n=4$) in Lemma \ref{GallaiED} and note that $k\le r(K_3, K_5)+\ell =16$, then it follows that $gr^3_2(K_3, K_9, K_k)\ge k+2$.
To prove the reverse inequality, consider a Gallai $3$-coloring of $K_{k+2}$ that avoids a red/blue $K_3$, a red/green $K_9$, and a blue/green $K_k$.  Let $X$ be the set of vertices that are incident with red edges and let $Y$ be the vertices that are not incident with red edges.  Note that all edges joining vertices in $X$ to vertices in $Y$ are green, all red edges are contained in the subgraph induced by $X$, and all blue edges are contained in the subgraph induced by $Y$.  

Since the subgraph induced by $Y\cup \{x\}$, for any $x\in X$, is a complete blue/green subgraph, assume that $|Y|\le k-2$.  It follows that $|X|\ge 4$, so $X$ must contain at least one green edge, which we denote by $x_1x_2$.  The subgraph induced by $Y\cup \{x_1,x_2\}$ is a blue/green complete subgraph, from which it follows that $|Y|\le k-3$ and $|X|\ge 5$.  Since the subgraph induced by $X\cup \{y\}$, for any $y\in Y$, is a red/green complete subgraph, assume that $|X|\le 7$ and $|Y|\ge k+2-7=k-5$.

If $|X|=5$ and $|Y|=k+2-5=k-5\ge 9$, then $r(K_3, K_4)=9$ \cite{GG} implies that the subgraph induced by $Y$ contains a green $K_4$.  The vertices in this green $K_4$, along with $X$, form a red/green $K_9$.   If $6\le |X|\le 7$ and $k-5\le |Y|\le k-4$, then $|Y|\ge 7$ and $r(K_3, K_3)=6$ \cite{GG} implies that there exists a green $K_3$ in $Y$.  The vertices of this $K_3$, along with $X$, form a red/green $K_9$.  It follows that $gr^3_2(K_3, K_9, K_k)\le k+2$.

\underline{Case 3:} Suppose that $17\le k\le 24$.  If we let $G\cong K_2$ (so that $\ell=1$ and $n=2$) in Lemma \ref{GallaiED} and note that $k\le r(K_3, K_7)+\ell =24$, then it follows that $gr^3_2(K_3, K_9, K_k)\ge k+1$. 
To prove the reverse inequality, consider a Gallai $3$-coloring of $K_{k+1}$ that avoids a blue/green $K_k$.  Then there exists a red matching of size at least two (otherwise, it is possible to delete a single vertex that removes all red edges, leaving a blue/green $K_k$).  Since $k+1\ge r(K_3, K_5)+4=18$, Lemma \ref{GallaiKn} implies that there exists a red/blue $K_3$, a red/green $K_9$, or a blue/green $K_k$.  It follows that $gr^3_2(K_3, K_9, K_k)\le k+1$.

\underline{Case 4:} Suppose that $25\le k\le 35$.  Since $r(K_3, K_9)=36$, there exists a blue/green $K_{k-1}$ that avoids a blue $K_3$ (and hence, a red/blue $K_3$) and a green $K_9$ (and hence, a red/green $K_9$).   It follows that $gr^3_2(K_3, K_9, K_k)\ge k$.  To prove the reverse inequality, consider a Gallai $3$-coloring of $K_k$ that avoids a blue/green $K_k$ (so, there exists at least one red edge).   Since $$k\ge r(K_3, K_7)+2=25,$$ Lemma \ref{GallaiKn} implies that there exists a red/blue $K_3$, a red/green $K_9$, or a blue/green $K_k$.  It follows that $gr^3_2(K_3, K_9, K_k)\le k$.
%
\end{proof}

Note that in many of the proofs in this section, Theorem \ref{rededge} could have been used to prove the upper bounds across ranges of $k$ once the first value in the range was determined.  Since our upper bound proofs worked across each range for $k$, this approach was not necesary.  In the next section, Theorem \ref{rededge} will play a more significant role in the proofs of upper bounds.

\section{The Chromatic Gallai-Ramsey Number $gr^3_2(K_4, K_j, K_k)$}\label{main3}

The value $r^3_2(K_4, K_4, K_4)=10$ was determined in \cite{CL}.  The following theorem provides a generalization of the Gallai analogue of this result.







\begin{theorem}\label{thm44k}
If $k\ge 4$, then $$gr^3_2(K_4, K_4, K_k)=\begin{cases} k+1 & \mbox{if $4\le k\le 8$} \\ k & \mbox{if $9\le k\le 17$} \\ 18 & \mbox{if $k\ge 18$.} \end{cases}$$
\end{theorem}

\begin{proof} Since $r(K_4, K_4)=18$ \cite{GG}, Theorem \ref{upperbound} implies that $$gr^3_2(K_4, K_4, K_k)=18 \quad \mbox{if $k\ge 18$.}$$ For $k<18$, we separate the proof into the following cases.

\underline{Case 1:} Suppose that $k=4$.  The Gallai $3$-coloring of $K_4$ shown in Figure \ref{r44k1} avoids a red/blue $K_4$, a red/green $K_4$, and a blue/green $K_4$.  
\begin{figure}[h!]
\centering
\begin{subfigure}{.29\textwidth}
    \centering
    \includegraphics[width=1\linewidth]{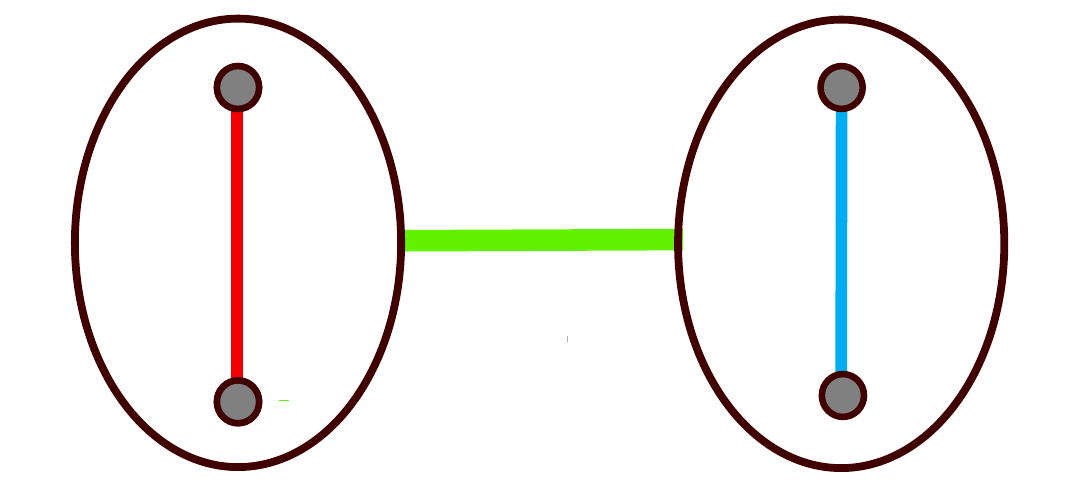}  
    \caption{}
    \label{r44k1}
\end{subfigure}\quad \ 
\begin{subfigure}{.29\textwidth}
    \centering
    \includegraphics[width=1\linewidth]{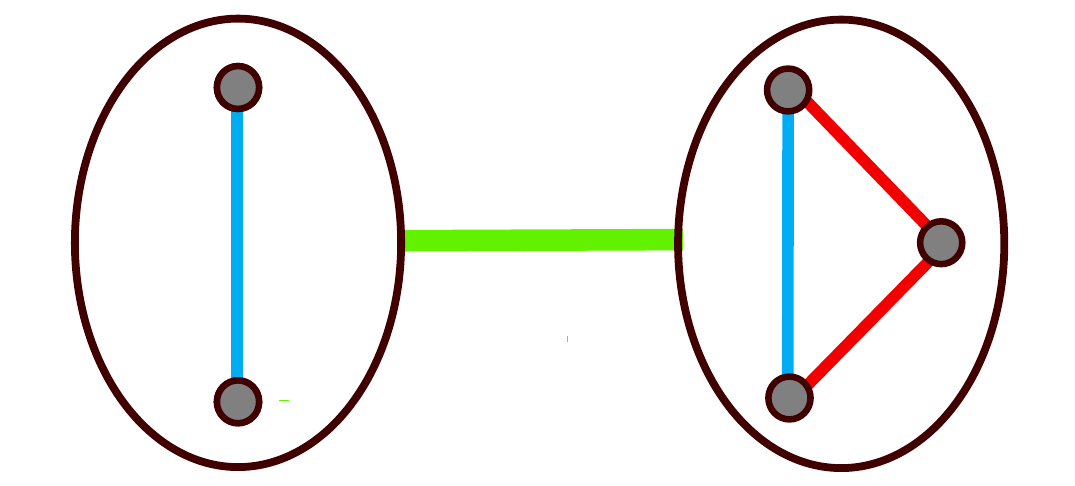}  
    \caption{}
    \label{r44k2}
\end{subfigure}\quad \ 
\begin{subfigure}{.29\textwidth}
    \centering
    \includegraphics[width=1\linewidth]{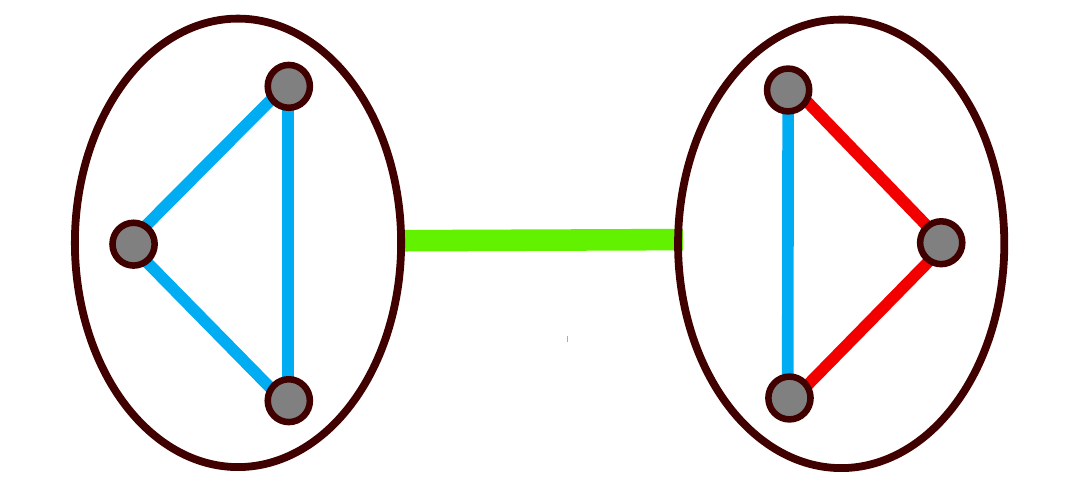}  
    \caption{}
    \label{r44k3}
\end{subfigure}\\
\begin{subfigure}{.28\textwidth}
    \centering
    \includegraphics[width=1\linewidth]{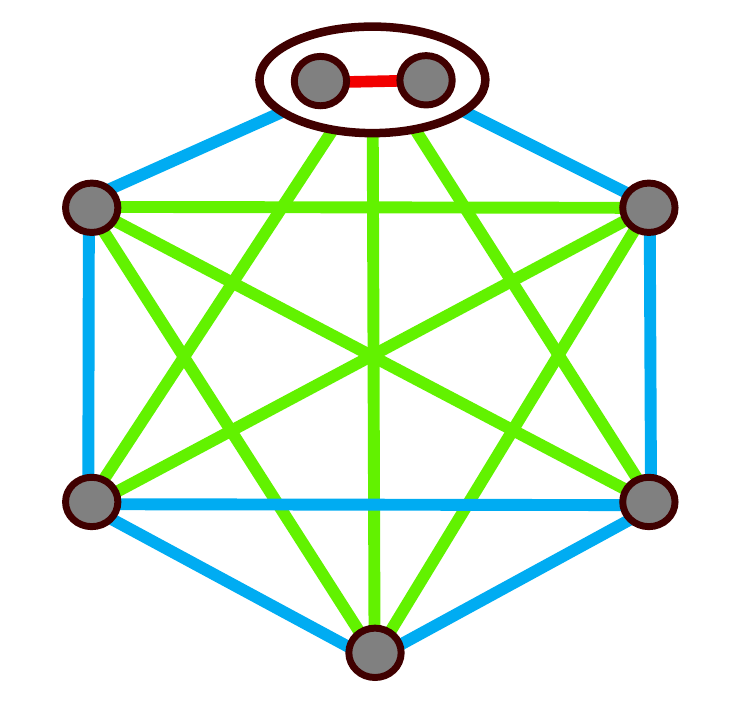}  
    \caption{}
    \label{r44k4}
\end{subfigure} \qquad
\begin{subfigure}{.29\textwidth}
    \centering
    \includegraphics[width=1\linewidth]{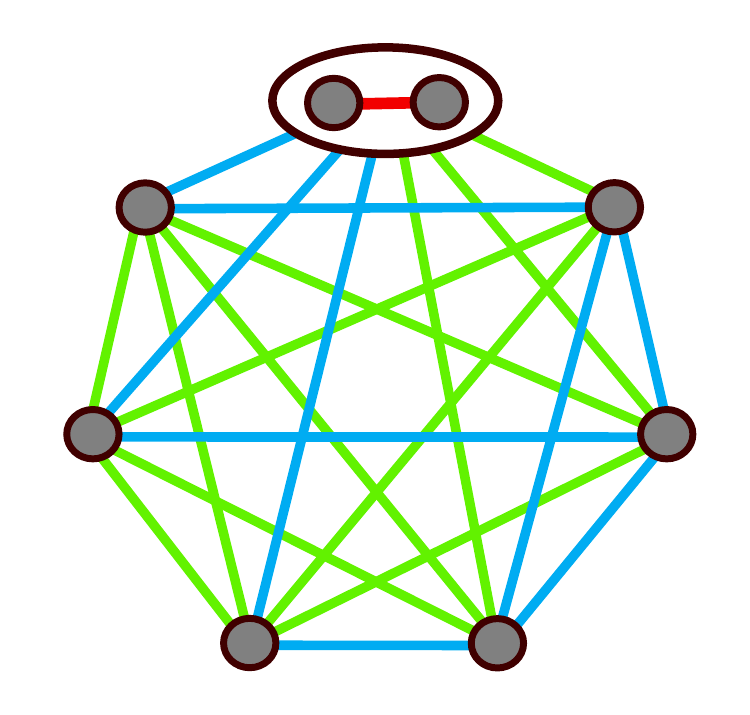}  
    \caption{}
    \label{r44k5}
\end{subfigure}
\caption{Critical colorings for $gr^3_2(K_4, K_4, K_k)$, when $4\le k\le 8$.}\label{r44k}
\end{figure}
It follows that $gr^3_2(K_4, K_4, K_4)\ge 5$.  To prove the reverse inequality, consider a Gallai $3$-coloring of $K_5$ and let $\mathcal{B}$ be the base graph, chosen to have minimal order.  If a $K_4$-subgraph whose edges use at most two colors is avoided, then Theorem \ref{Gallaistruct} and Lemma \ref{not3} imply that $|V(\mathcal{B})|=2$.   Suppose that the edges joining the two blocks are red.   We must consider two subcases.

\underline{Subcase 1.1:} Assume that one block contains a single vertex, which we denote by $x$,  and the other block has order $4$, with vertex set $Y$.  Then Lemma \ref{increaseorder} implies that the subgraph induced by $Y$ does not contain any red edges (as such an edge, along with $x$, would form a red $K_3$, resulting in a red/blue $K_4$ or a red/green $K_4$).  So, the subgraph induced by $Y$ forms a blue/green $K_4$.

\underline{Subcase 1.2:} Assume that one block has order $2$ and the other block has order $3$.   Denote the vertex sets for the blocks by $X=\{x_1, x_2\}$ and $y=\{y_1, y_2, y_3\}$.  Once again, we note that the subgraphs induced by $X$ and $Y$ cannot contain any red edges.  Without loss of generality, suppose that $x_1x_2$ is blue.  If any edge in $Y$ is blue (say, $y_1y_2$ is blue), then the subgraph iniduced by $\{x_1, x_2, y_1, y_2\}$ is a red/blue $K_4$.  So, all of the edges in $Y$ must be green, and the subgraph induced by $\{x_1, y_1, y_2, y_3\}$ is a red/green $K_4$.  In all cases, there exists a $K_4$-subgraph whose edges use at most two colors.  It follows that $gr^3_2(K_4, K_4, K_4)=5$.

\underline{Case 2:} Suppose that $5\le k\le 8$.  The Gallai $3$-colorings of $K_{k}$ shown in Figures \ref{r44k2}, \ref{r44k3}, \ref{r44k4}, and \ref{r44k5} provide the lower bounds $gr^3_2(K_4, K_4, K_k)\ge k+1$ for $k=5, 6, 7,$ and $8$, respectively.  The reverse inequality follows from induction on $4\le k\le 8$, with the base case corresponding to Case 1 of this proof.  Suppose that $gr^3_2(K_4, K_4, K_{k-1})\le k$ and consider a Gallai $3$-coloring of $K_{k+1}$ that avoids a red/blue $K_4$ and a red/green $K_4$.  By Theorem \ref{rededge}, every vertex is incident with a red edge, or there exists a red/blue $K_4$, a red/green $K_4$, or a blue/green $K_k$.  Assume the former case and also note that by the inductive hypothesis, there exists a blue/green $K_{k-1}$, whose vertex set we denote by $X:=\{x_1,x_2,\dots , x_{k-1}\}$.  Denote the other two vertices by $y$ and $z$.  

By the inductive hypothesis, the subgraph induced by  $\{x_2, \dots , x_{k-1}, y,z\}$ must contain a blue/green $K_{k-1}$, which must necessarily include one of $y$ or $z$.  Without loss of generality, suppose that $\{x_2, \dots , x_{k-1}, y\}$ induces a blue/green $K_{k-1}$.  If the subgraph induced by $\{x_1, x_2, \dots , x_{k-1},y\}$ is not a blue/green $K_k$, then $x_1y$ must be red and $y$ joins to $\{x_1, x_2, \dots , x_{k-1}\}$ with only blue and green edges.  It follows that $x_2z, \dots , x_{k-1}z$ are all red edges.
Without loss of generality suppose that $yz$ is red or blue.  If $yz$ is red, then depending on the color of $x_1x_2$ (which is blue or green), the subgraph induced by $\{x_1,x_2,y,z\}$ is a red/blue or red/green $K_4$.  So, suppose that $yz$ is blue.  If $x_1x_2$ is also blue, then the subgraph induced by $\{x_1,x_2,y,z\}$ is a red/blue $K_4$.  So, assume that $x_1x_2$ is green.  Then $x_1y$ must be red, contradicting the fact that $\{x_2, \dots , x_{k-1}, y\}$ induces a blue/green $K_{k-1}$.  It follows that $gr^3_2(K_4, K_4, K_k)\le k+1$.

\underline{Case 3:} Suppose that $k=9$.  Since $r(K_4, K_4)=18$ \cite{GG}, there exists a blue/green $K_{8}$ that avoids a blue $K_4$ (and hence, a red/blue $K_4$) and a green $K_4$ (and hence, a red/green $K_4$).   As there are not enough vertices to form a blue/green $K_9$, it follows that $gr^3_2(K_4, K_4, K_9)\ge 9$.  To prove the reverse inequality, consider a Gallai $3$-coloring of $K_9$ that avoids a red/blue $K_4$, a red/green $K_4$, and a blue/green $K_9$.  Let $\mathcal{B}$ be the base graph, chosen to have minimal order.  Theorem \ref{Gallaistruct} and Lemma \ref{not3} imply that $2\le |V(\mathcal{B})|\le 8$ and $|V(\mathcal{B})|\ne 3$.  Consider the following subcases.

\underline{Subcase 3.1:} Suppose that $|V(\mathcal{B})|=2$.  If the edge in $\mathcal{B}$ is red, then assume that neither block contains a red edge, as this would result in a red $K_3$, and hence, a red/blue $K_4$ or a red/green $K_4$ by Lemma \ref{increaseorder}.  There are two subsubcases to consider.

\underline{Subsubcase 3.1.1:} Suppose that one block in $\mathcal{B}$ contains only one vertex, which we denote by $x$.   Let $Y$ be the vertex set for the other block.  If the edge in $\mathcal{B}$ is red, then $r(K_3, K_3)=6$ \cite{GG} implies that the subgraph induced by $Y$ contains a red/blue $K_3$ or a green $K_3$.  Such a subgraph, along with $x$, forms a red/blue $K_4$ or a red/green $K_4$, respectively.  So, assume that the edge in $\mathcal{B}$ is blue or green (without loss of generality, suppose it is blue).  By Theorem \ref{gr34k}, $gr^3_2(K_3, K_4, K_8)=8$, from which it follows that the subgraph induced by $Y$ contains a red/blue $K_3$, a red/green $K_4$, or a blue/green $K_8$.  In the first case, including $x$ results in a red/green $K_4$, and in the last case, including $x$ results in a blue/green $K_9$. 

\underline{Subsubcase 3.1.2:} Suppose that both blocks have order at least $2$.  If the edge in $\mathcal{B}$ is red, then no red edge exists in either block as such an edge would result in a red $K_3$, and by Lemma \ref{increaseorder}, there would exist a red/blue $K_4$ or a red/green $K_4$.  If both blocks contain a blue edge or if both blocks contain a green edge, then such a pair of edges results in a red/blue $K_4$ or a red/green $K_4$, respectively.  Hence, one block must form a red/blue complete graph and the other must form a red/green complete graph.  By the pigeonhole principle, one block contains at least five vertices forming a blue $K_5$ or a green $K_5$.

If the edge in $\mathcal{B}$ is blue or green (assume it is blue), then at most one of the blocks can contain a red or a blue edge; otherwise, a red/blue $K_4$ can be formed.  Hence, one block is a complete green graph., which has order $2$ or $3$.  If its order is $2$, then the other block has order $7$, and by Theorem \ref{gr34k}, $gr^3_2(K_3, K_4, K_7)=7$.  So, it contains a red/blue $K_3$, a red/green $K_4$, or a blue/green $K_7$.  In the first case, including one vertex from the smaller block results in a red/blue $K_4$.  In the last case, including both of the vertices from the smaller block results in a blue/green $K_9$.  If the complete green block has order $3$, then the other block has order $6$, and by Theorem \ref{gr34k}, $gr^3_2(K_3, K_4, K_6)=6$.  So, it contains a red/blue $K_3$, a red/green $K_4$, or a blue/green $K_6$.  In the first case, including one vertex from the smaller block results in a red/blue $K_4$.  In the last case, including all three of the vertices from the smaller block results in a blue/green $K_9$. 

\underline{Subcase 3.2:} Suppose that $4\le |V(\mathcal{B})|\le 8$.  If the edges in $\mathcal{B}$ are red/blue or red/green, then a red/blue $K_4$ or a red/green $K_4$ exists.  So, assume that $\mathcal{B}$ is a blue/green complete graph.  Also, note that every vertex in $\mathcal{B}$ is incident with both a blue edge and a green edge; otherwise, the Gallai coloring can be formed with a base graph of order $2$.  Some block must contain a red edge, or the Gallai coloring is a blue/green $K_k$.  Also, at most one block contains a red edge; otherwise, a red/blue $K_4$ or a red/green $K_4$ can be formed.  Let $xy$ be the unique red edge and denote the other seven vertices by $z_1,z_2,\dots , z_7$.  Since rainbow triangles are avoided, $xz_i$ and $yz_i$ are both blue or both green, for each $i\in \{1, 2, \dots , 7\}$.  By the pigeonhole principle, at least four of the vertices in $\{z_1, z_2, \dots , z_7\}$ join to $x$ and $y$ with the same color.  Without loss of generality, suppose that $xz_1$, $yz_1$, $xz_2$, $yz_2$, $xz_3$, $yz_3$, $xz_4$, and $yz_4$ are all blue.  If any edge in the subgraph induced by $\{z_1,z_2,z_3,z_4\}$ is blue, then that edge, along with $x$ and $y$, forms a red/blue $K_4$.  Thus, $\{z_1,z_2,z_3,z_4\}$ induces a green $K_4$.

In both subcases, we find that there exists a red/blue $K_4$, a red/green $K_4$, or a blue/green $K_9$.  It follows that $gr^3_2(K_4, K_4, K_9)\le 9$.





\underline{Case 4:} Suppose that $10\le k\le 17$.  Since $r(K_4, K_4)=18$ \cite{GG}, there exists a blue/green $K_{k-1}$ that avoids a blue $K_4$ (and hence, a red/blue $K_4$) and a green $K_4$ (and hence, a red/green $K_4$).   As there are not enough vertices to form a blue/green $K_k$, it follows that $gr^3_2(K_4, K_4, K_k)\ge k$.  The reverse inequality follows from induction on $9\le k\le 17$, with the base case corresponding to Case 3 of this proof.  Suppose that $gr^3_2(K_4, K_4, K_{k-1})\le k-1$ and consider a Gallai $3$-coloring of $K_{k}$ that avoids a red/blue $K_4$ and a red/green $K_4$.  By Theorem \ref{rededge}, every vertex is incident with a red edge, or there exists a red/blue $K_4$, a red/green $K_4$, or a blue/green $K_k$.   Assume the former case and also note that by the inductive hypothesis, there exists a blue/green $K_{k-1}$, whose vertex set we denote by $X:=\{x_1,x_2,\dots , x_{k-1}\}$.  Denote the other vertex by $y$.  The only way for every vertex to be incident with a red edge is if $yx_i$ is red for all $i\in \{1, 2, \dots , k-1\}$.  The subgraph induced by $\{x_1, x_2, \dots , x_{k-1}\}$ is a blue/green complete graph of order $k-1\ge 9$, so it must contain a blue $K_3$ or a green $K_3$ since $r(K_3, K_3)=6$ \cite{GG}.  Such a $K_3$-subgraph, along with $y$, forms a red/blue $K_4$ or a red/green $K_4$, respectively.  It follows that $gr^3_2(K_4, K_4, K_k)\le k$.
\end{proof}


\begin{theorem}
If $k\ge 5$, then $$gr^3_2(K_4, K_5, K_k)=\begin{cases} k+2 & \mbox{if $5\le k\le 9$} \\ k+1 & \mbox{if $10\le k\le 14$} \\ k & \mbox{if $15\le k\le 24$} \\ 25 & \mbox{if $k\ge 25$.} \end{cases}$$
\end{theorem}

\begin{proof} Since $r(K_4, K_5)=25$ \cite{Ka1, MR4}, Theorem \ref{upperbound} implies that $$gr^3_2(K_4, K_5, K_k)=25 \quad \mbox{if $k\ge 25$.}$$ For $k<25$, we separate the proof into the following cases.

\underline{Case 1:} Suppose that $k=5$.  The Gallai $3$-coloring of $K_6$ shown in Figure \ref{r45k1} avoids a red/blue $K_4$, a red/green $K_5$, and a blue/green $K_5$.  
\begin{figure}[h!]
\centering
\begin{subfigure}{.26\textwidth}
    \centering
    \includegraphics[width=1\linewidth]{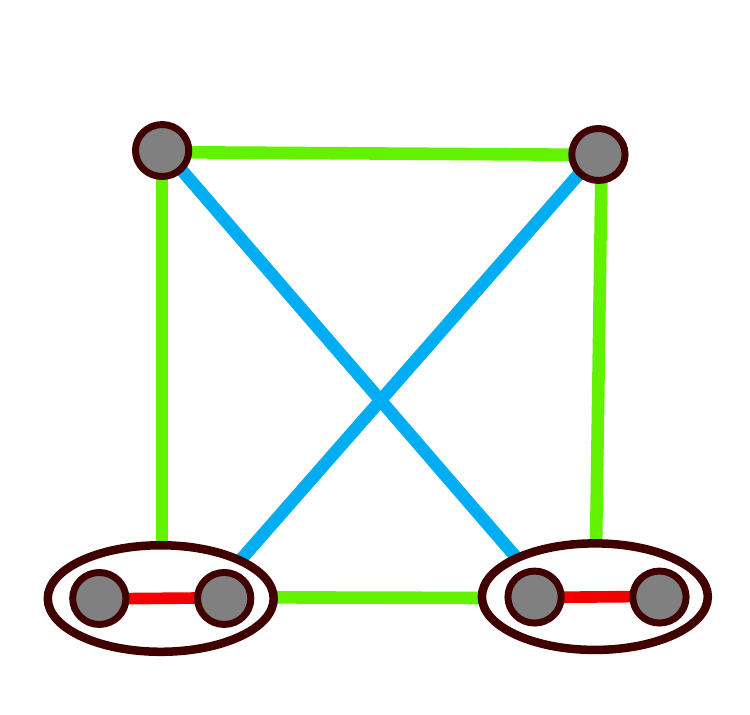}  
    \caption{}
    \label{r45k1}
\end{subfigure}\qquad
\begin{subfigure}{.27\textwidth}
    \centering
    \includegraphics[width=1\linewidth]{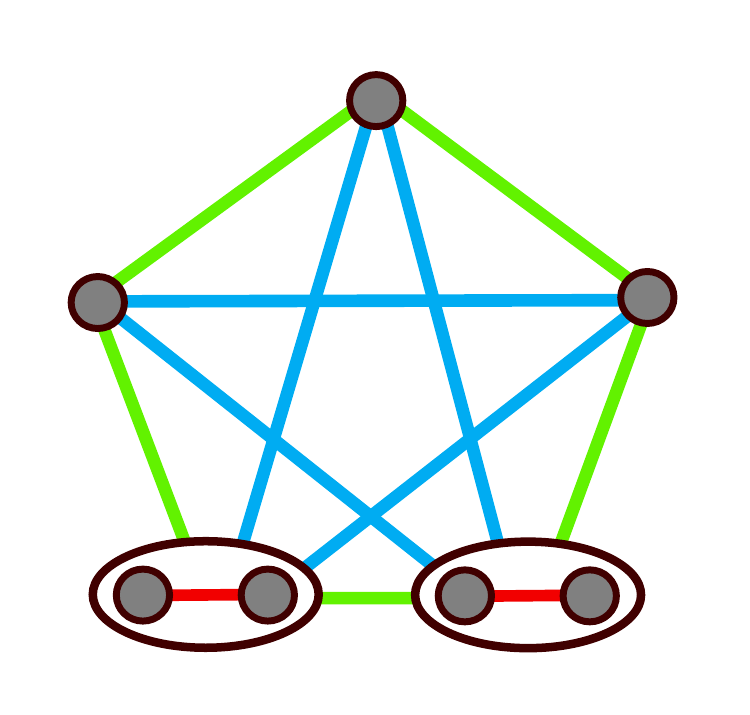}  
    \caption{}
    \label{r45k2}
\end{subfigure}\qquad 
\begin{subfigure}{.27\textwidth}
    \centering
    \includegraphics[width=1\linewidth]{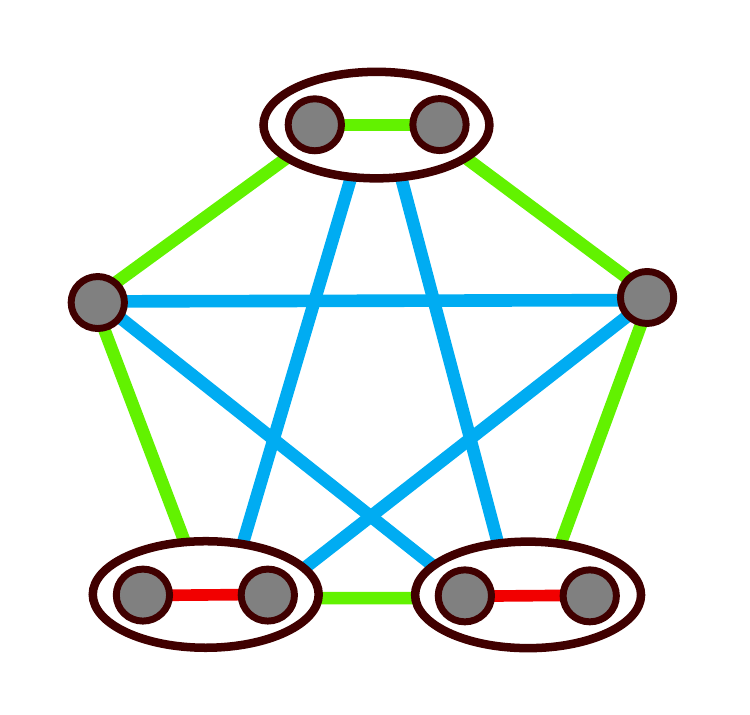}  
    \caption{}
    \label{r45k3}
\end{subfigure}\\
\begin{subfigure}{.27\textwidth}
    \centering
    \includegraphics[width=1\linewidth]{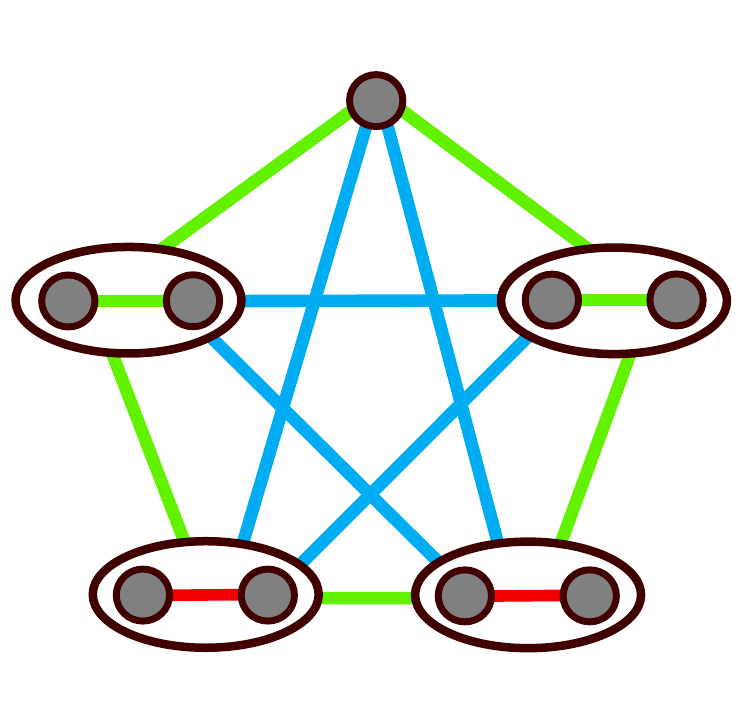}  
    \caption{}
    \label{r45k4}
\end{subfigure} \qquad
\begin{subfigure}{.27\textwidth}
    \centering
    \includegraphics[width=1\linewidth]{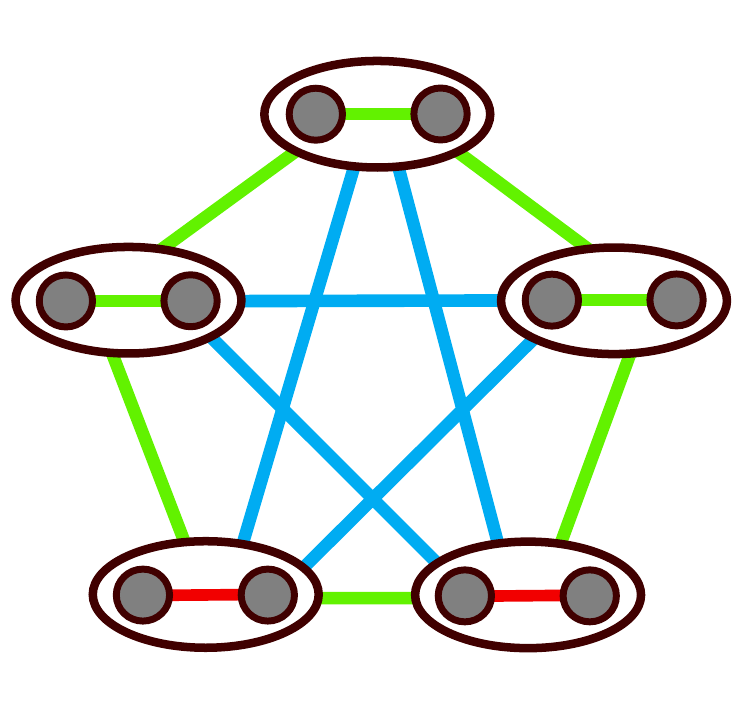}  
    \caption{}
    \label{r45k5}
\end{subfigure}
\caption{Critical colorings for $gr^3_2(K_4, K_5, K_k)$, when $5\le k\le 9$.}\label{r45k}
\end{figure}
It follows that $gr^3_2(K_4, K_5, K_5)\ge 7$.  To prove the reverse inequality, consider a Gallai $3$-coloring of $K_7$, and let $\mathcal{B}$ be the base graph, chosen to have minimal order.  By Theorem \ref{Gallaistruct} and Lemma \ref{not3}, $|V(\mathcal{B})|\in \{2, 4\}$.  

\underline{Subcase 1.1:}  Suppose that $|V(\mathcal{B})|=2$.  First, consider the possibility that the edge in $\mathcal{B}$ is red or blue (without loss of generality, suppose that it is red).  
By Theorem \ref{gr34k}, $gr^3_2(K_3, K_4, K_5)=5$.  So, if a block has order at least $5$, then it contains a red/blue $K_3$, a red/green $K_4$, or a blue/green $K_5$.  In the first case, including a single vertex from the smaller block results in a red/blue $K_4$.  In the second case, including a single vertex from the smaller block results in a red/green $K_5$.  Thus, assume that one block has order $4$ and the other block has order $3$, and we denote their vertex sets by $X$ and $Y$, respectively.  If the subgraph induced by $Y$ is a red/blue $K_3$, then including a vertex from $X$ results in a red/blue $K_4$.  If the subgraph induced by $Y$ is a red/green $K_3$, then if any red or blue edge exists in the subgraph induced by $X$, that edge, along with the vertices in $Y$, forms a red/green $K_5$.  So, assume that the subgraph induced by $Y$ is a blue/green $K_3$.  If the subgraph induced by $Y$ contains a blue edge, then the subgraph induced by $X$ cannot contain a red or blue edge without forming a red/blue $K_4$.  So, assume that the subgraph induced by $Y$ is a green $K_3$.  If the subgraph induced by $X$ contains a red or green edge, then such an edge, along with $Y$, forms a red/green $K_5$.  So, assume that all of the edges in the subgraph induced by $X$ are blue, forming a blue $K_4$.

Now consider the case where the edge in $\mathcal{B}$ is green.  By Theorem \ref{thm44k}, $gr^3_2(K_4, K_4, K_4)=5$.  So, if a block has order at least $5$, then it contains a $K_4$ whose edges use at most two colors.  If such a $K_4$ is red/green or blue/green, then including a single vertex from the other block results in a red/green $K_5$ or a blue/ green $K_5$, respectively.  The only other possibility is that it contains a red/blue $K_4$.  Thus, assume that one block has order $4$ and the other block has order $3$, and we denote their vertex sets by $X$ and $Y$, respectively.  
If $X$ contains a red/green $K_3$, then $Y$ cannot contain a red or green edge without resulting in a red/green $K_5$.  So, suppose that $Y$ induces a blue $K_3$.  Then if $X$ contains a blue or green edge, that edge, along with $Y$, forms a blue/green $K_5$.  So, $X$ must induce a red $K_4$.  
If $X$ contains a blue/green $K_3$, then $Y$ cannot contain a blue or green edge without resulting in a blue/green $K_5$.  So, suppose that $Y$ induces a red $K_3$.  Then if $X$ contains a red or green edge, that edge, along with $Y$, forms a red/green $K_5$.  So, $X$ must induce a blue $K_4$.
Finally, consider the case where $X$ does not contain a red/green $K_3$ or a blue/green $K_3$.  Then the subgraph induced by $X$ must be a red/blue $K_4$.  



\underline{Subcase 1.2:} Suppose that $|V(\mathcal{B})|=4$.  If the edges in $\mathcal{B}$ are red and blue, then a red/blue $K_4$ can be formed by selecting a single vertex from each block.  So, assume that the edges in $\mathcal{B}$ are red/green or blue/green (without loss of generality, assume that they are red/green).  If any block contains a red or green edge, then that edge, along with a single vertex from each of the other blocks, forms a red/green $K_5$.  So the edges in all of the blocks must be blue.  Note that every vertex in $\mathcal{B}$ is incident with both red and green edges, otherwise, the Gallai $3$-coloring corresponds with a base graph of order $2$, contradicting the minimality of the order of $\mathcal{B}$.  If any block has order at least three, then selecting a single vertex from a block that is red-adjacent to the block of order at least $3$ results in a red/blue $K_4$.  So, three blocks have order $2$ and one block has order $1$.  Avoiding a red/blue $K_4$ also forces the blocks of order $2$ to all be green-adjacent, forming a blue/green $K_6$.  

In both subcases, we find that there exists a red/blue $K_4$, a red/green $K_5$, or a blue/green $K_5$.  It follows that $gr^3_2(K_4, K_5, K_5)\le 7$.



\underline{Case 2:} Suppose that $6\le k\le 9$.  The Gallai $3$-colorings of $K_{k+1}$ given in Figures \ref{r45k2}, \ref{r45k3}, \ref{r45k4}, and \ref{r45k5} imply the lower bound $$gr^3_2(K_4, K_5, K_k)\ge k+2$$ when $k=6$, $7$, $8$, and $9$, respectively.  In each case, the colorings avoid a red/blue $K_4$, a red/green $K_5$, and a blue/green $K_k$. 

To prove the reverse inequality, we use induction on $5\le k\le 9$, with Case 1 of this proof serving as the base case.  For $6\le k\le 9$, assume that $gr^3_2(K_4, K_5, K_{k-1})\le k+1$, and consider a Gallai $3$-coloring of $K_{k+2}$ that avoids a red/blue $K_4$ and a red/green $K_5$.  By the inductive hypothesis, there exists a blue/green $K_{k-1}$.  Also, by Theorem \ref{rededge}, every vertex must be incident with a red edge.  Let $\mathcal{B}$ be the base graph for this Gallai $3$-coloring of $K_{k+2}$, chosen to have minimal order.  By Theorem \ref{Gallaistruct} and Lemma \ref{not3}, if a blue/green $K_k$  is avoided, then $|V(\mathcal{B})|\in \{2, 4, 5,\dots , k-1\}$.

\underline{Subcase 2.1:} Suppose that $|V(\mathcal{B})|=2$.   First, consider the possibility that the edge in $\mathcal{B}$ is red.  Then it is not possibile for both blocks to contain a red or blue edge as two such edges would produce a red/blue $K_4$.  Denote the vertex sets for the blocks by $X$ and $Y$, and assume that $Y$ does not contains any red or blue edges.  If $|Y|\ge 4$, then selecting $x\in X$, the subgraph induced by $\{x\}\cup Y$ is a red/green $K_5$.  So, assume that $|Y|\le 3$.  If $|Y|=3$, then if $X$ contains a red or green edge, that edge, along with $Y$, forms a red/green $K_5$.  If $X$ does not contain any red or green edges, then the subgraph induced by $X$ is a blue $K_{k-1}$, with $k-1\ge 5$.  If $|Y|=2$, then $|X|=k\ge 6$.  Since $r(K_3, K_3)=6$ \cite{GG}, the subgraph induced by $X$ contains a red/green $K_3$ or a blue $K_3$.  In the former case, the red/green $K_3$, along with the vertices in $Y$, forms a red/green $K_4$.  In the latter case, the blue $K_3$, along with any vertex from $Y$, forms a red/blue $K_4$.  If $|Y|=1$, then $|X|=k+1\ge 7$.  By Theorem \ref{thm44k}, $gr^3_2(K_4, K_4, K_{k})\le k+1$, for all $6\le k\le 9$, so the subgraph induced by $X$ contains a red/blue $K_4$, a red/green $K_4$, or a blue/green $K_k$.  In the second case, a red/green $K_4$, along with a vertex from $Y$, results in a red/green $K_5$.

Next, suppose that the edge in $\mathcal{B}$ is blue.  In order for every vertex to be incident with a red edge, each block must contain a red edge.  The vertices for two such edges induce a red/blue $K_4$. 

Finally, suppose that the edge in $\mathcal{B}$ is green.  By the pigeonhole principle, one block must contain at least $\ceil{\frac{k+2}{2}}\ge 4$ vertices.  Without loss of generality, suppose that the vertex sets for the blocks are denoted by $X$ and $Y$, and $|X|\ge 4$.  In order for every vertex to be incident with a red edge, each block must contain a red edge, so it follows that $|Y|\ge 2$.  If $X$ contains a red/green $K_3$, then that $K_3$, along with a red edge from $Y$, forms a red/green $K_5$.  So, assume that $X$ does not contain any red/green $K_3$-subgraphs.  If $X$ contains two disjoint red edges, then those edges must be joined by blue edges, forming a red/blue $K_4$.  If no two disjoint red edges exist in $X$, then in order for every vertex to be incident with a red edge, $X$ must contain a red $K_{1,3}$.   The leaves of such a $K_{1,3}$ can only be joined by blue edges, resulting in a red/blue $K_4$.

\underline{Subcase 2.2:} Suppose that $|V(\mathcal{B})|=4$.  If the edges in $\mathcal{B}$ are red and blue, then a red/blue $K_4$ can be formed by selecting a single vertex from each block.  If the edges in $\mathcal{B}$ are red and green, then no red or green edge exists in any of the blocks without producing a red/green $K_5$.   So, assume that all blocks form complete blue subgraphs.  By the minimality of $\mathcal{B}$, all vertices in $\mathcal{B}$ are incident with both red and green edges.  So, if any block contains three or more vertices, then that block, along with a single vertex from another red-adjacent block, forms a red/blue $K_4$.  Hence, we assume that each block has order at most $2$, from which it follows that $k+2\le 8$ (i.e., $k=6$).  As each block must then have order exactly two, the vertices in two red-adjacent blocks will form a red/blue $K_4$.

Assume that the edges in $\mathcal{B}$ are blue and green.  Since $\mathcal{B}$ is chosen to have minimal order, it follows that every vertex in $\mathcal{B}$ is incident with both blue and green edges.  By Therorem \ref{rededge}, every vertex must be incident with a red edge.  So, each block has order at least two and contains a red edge.  Selecting two red edges from distinct blocks that are joined by blue edges results in a red/blue $K_4$.

\underline{Subcase 2.3:} Suppose that $|V(\mathcal{B})|\in \{5, 6, 7, 8\}$.  Then the edges in $\mathcal{B}$ are blue and green.  Since $\mathcal{B}$ is chosen to have minimal order, it follows that every vertex in $\mathcal{B}$ is incident with both blue and green edges.  If a block only contains a single vertex, then that vertex is not incident with any red edges, contradicting Theorem \ref{rededge}.  So, each block has order at least two and contains a red edge.  Selecting two red edges from distinct blocks that are joined by blue edges results in a red/blue $K_4$.  

\underline{Case 3:} Suppose that $k=10$.  While the Gallai $3$-coloring of $K_{10}$ given in Figure \ref{r45k5} provides the lower bound $gr^3_2(K_4, K_5, K_{10})\ge 11$, it is not unique.  The Gallai $3$-coloring of $K_{10}$ shown in Figure \ref{r45{10}} avoids a red/blue $K_4$, a red/green $K_5$, and a blue/green $K_{10}$ (here, the block consisting of blue and green edges is a critical coloring for $r(K_4, K_3)$ in colors blue and green), and it sets us up for the lower bound construction we will use in the next case.
\begin{figure}[h!]
\centerline{
\includegraphics[width=0.6\textwidth]{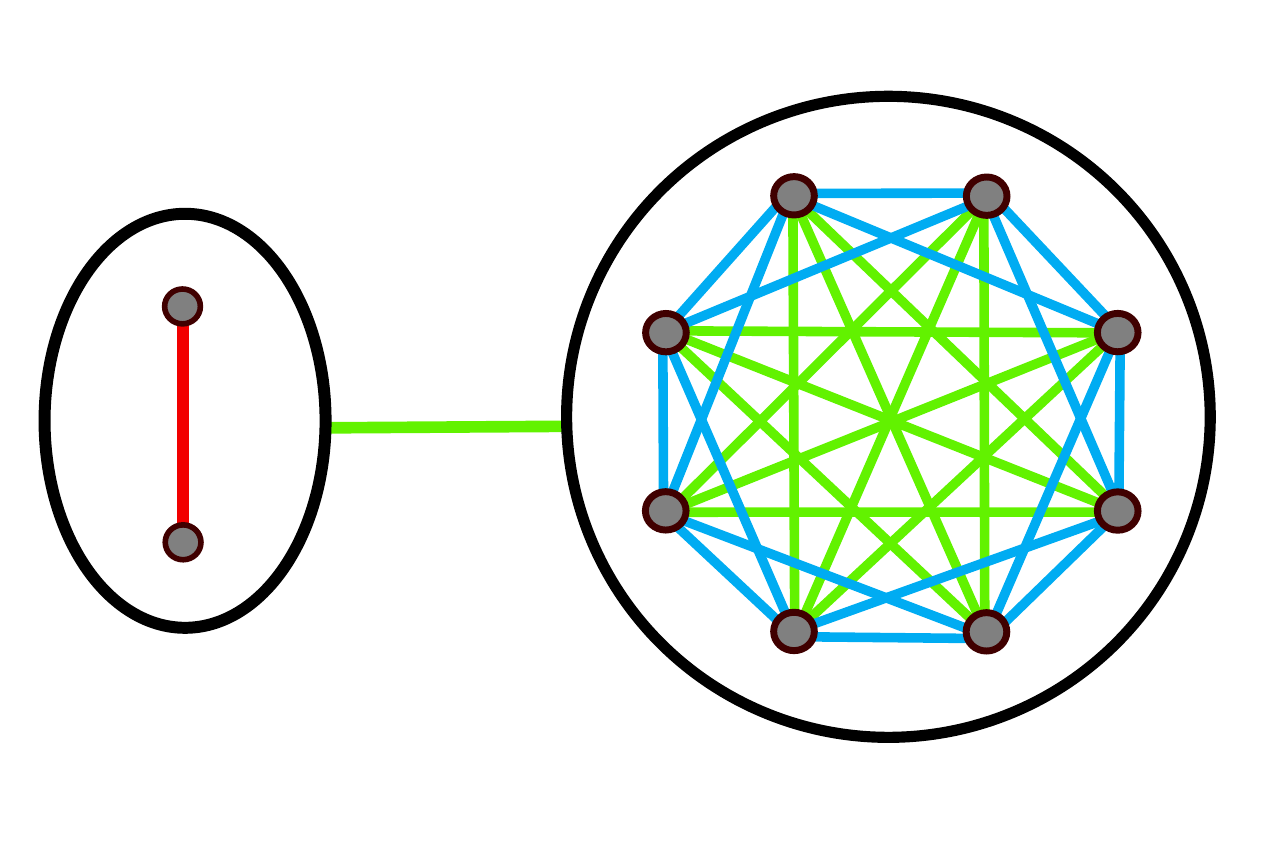}}
\caption{A critical coloring for $gr^3_2(K_4, K_5, K_{10})$.}\label{r45{10}}
\end{figure}

To prove that $gr^3_2(K_4, K_5, K_{10})\le 11$, consider a Gallai $3$-coloring of $K_{11}$ that avoids a red/blue $K_4$, a red/green $K_5$, and a blue/green $K_{10}$.  A red edge must exist, otherwise, it is a blue/green $K_{11}$.  Let $x$ be a vertex that is incident with a red edge.  Denote by $N_r(x)$, $N_b(x)$, and $N_g(x)$ the red, blue, and green neighborhoods of $x$, respectively.  It is possible that $N_b(x)$ or $N_g(x)$ could be empty, but $|N_r(x)|\ge 1$ by assumption.  Since rainbow triangles are avoided, all edges joining $N_r(x)$ and $N_b(x)$ are red or blue.  If either of $N_r(x)$ or $N_b(x)$ contain a red or blue edge, then that edge, along with $x$ and a single vertex from the other set, forms a red/blue $K_4$.  So, if $N_r(x)$ contains any edges, they must be green (the same is true for $N_b(x)$).  Also, avoiding a rainbow triangle forces all edges joining $N_r(x)$ and $N_g(x)$ to be red or green. Consider two subcases.

\underline{Subcase 3.1:} Suppose that $|N_r(x)|=1$ and let $y$ be the vertex in $N_r(x)$.  Since the subgraph induced by $N_b(x)$ only contains green edges, and a green $K_5$ is avoided, it follows that $|N_b(x)|\le 4$ and $|N_g(x)|\ge 5$.  Note that a red edge must exist in the subgraph induced by $N_g(x)$ (otherwise, the deletion of $y$ would result in a blue/green $K_{10}$).  Let $z_1z_2$ be a red edge such that $z_1,z_2\in N_g(x)$ and let $z_3,z_4,z_5\in N_g(x)$ be three other vertices.  If $z_iz_j$ is green for any $i\in \{1,2\}$ and $j\in \{3,4,5\}$, then the subgraph induced by $\{z_1,z_2,z_j\}$ is a red/green $K_3$ that, along with $x$ and $y$, forms a red/green $K_5$.  So, all edges if the form $z_iz_j$ are red or blue.  If any edge in the subgraph induced by $\{z_3,z_4,z_5\}$ is red or blue, then that edge, along with $z_1$ and $z_2$, forms a red/blue $K_4$.  Hence, the subgraph induced by $\{z_3,z_4, z_5\}$ is a green $K_3$ that, along with $x$ and $y$, forms a red/green $K_5$.

\underline{Subcase 3.2:} Suppose that $|N_r(x)|\ge 2$.  If the subgraph induced by $N_g(x)$ contains a red or green edge, then that edge, along with $x$ and a green edge in $N_r(x)$, forms a red/green $K_5$. So, $N_g(x)$ only contains blue edges and $|N_g(x)|\le 3$ in order to avoid a blue $K_4$.  Since $N_r(x)$ only contains green edges, if $|N_r(x)|\ge 3$, then selecting three vertices from $N_r(x)$, one vertex from $N_g(x)$, and vertex $x$, results in a red/green $K_5$.  It follows that $|N_r(x)|\le 2$, which implies that $|N_b(x)|\ge 5$.  Since $N_b(x)$ only contains green edges, it contains a green $K_5$.

In both subcases, we find that there exists a red/blue $K_4$, a red/green $K_5$, or a blue/green $K_{10}$.  It follows that $gr^3_2(K_4, K_5, K_{10})\le 11$.

\underline{Case 4:} Suppose that $11\le k\le 14$.  For the lower bound, we start with the construction given in Figure \ref{r45{10}}, and add in a green $K_{k-10}$, as shown in Figure \ref{constrcriticoloring}.  
\begin{figure}[h!]
\centerline{
\includegraphics[width=0.8\textwidth]{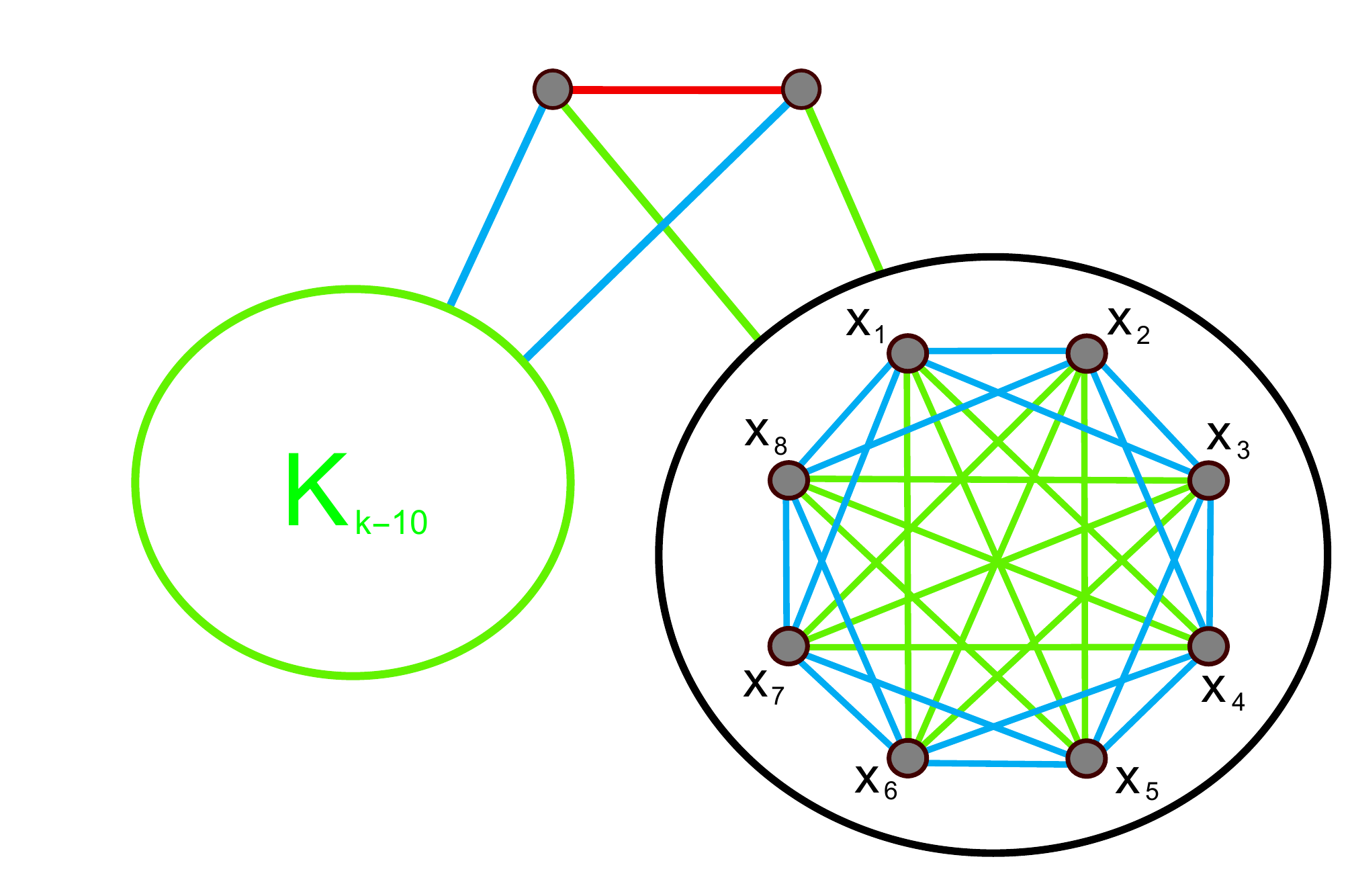}}
\caption{Constructing a critical coloring for $gr^3_2(K_4, K_5, K_{k})$, when $11\le k\le 14$.}\label{constrcriticoloring}
\end{figure}
Here, the subgraph induced by $\{x_1, x_2, \dots , x_8\}$ is a critical coloring for $r(K_4, K_3)$ (in colors blue and green).  Suppose that the vertex set for the green $K_{k-10}$-subgraph is given by $\{y_1, y
_2, \dots , y_{k-10}\}$, and the edges $x_iy_j$ are colored blue if $i$ and $j$ have the same parity and green if they have different parities.  The resulting Gallai $3$-coloring of $K_{k}$ avoids a red/blue $K_4$, a red/green $K_5$, and a blue/green $K_{k}$.  It follows that $gr^3_2(K_4, K_5, K_k)\ge k+1$.

To prove the reverse inequality, we use induction on $10\le k\le 14$, with Case 3 of this proof serving as the base case.  For $11\le k\le 14$, assume that $gr^3_2(K_4, K_5, K_{k-1})\le k$ and consider a Gallai $3$-coloring of $K_{k+1}$ that avoids a red/blue $K_4$ and a red/green $K_5$.  By the inductive hypothesis, there exists a blue/green $K_{k-1}$.  Note that $10\le k-1\le 13$, from which $r(K_3, K_4)=9$ implies that the blue/green $K_{k-1}$ contains a blue $K_3$ or a green $K_4$.  By Theorem \ref{rededge}, every vertex must be incident with at least one red edge.   If $\{x_1,x_2,x_3\}$ induces a blue $K_3$ and $yx_1$ is red, then the subgraph induced by $\{x_1,x_2,x_3,y\}$ is a red/blue $K_4$.  If $\{x_1,x_2,x_3,x_4\}$ induces a green $K_4$ and $yx_1$ is red, then the subgraph induced by $\{x_1,x_2,x_3,x_4,y\}$ is a red/green $K_5$.  It follows that $gr^3_2(K_4, K_5, K_k)\le k+1$.

\underline{Case 5:} Suppose that $15\le k\le 24$.  Since $r(K_4, K_5)=25$ \cite{GG}, there exists a blue/green $K_{k-1}$ that avoids a blue $K_4$ (and hence, a red/blue $K_4$) and a green $K_5$ (and hence, a red/green $K_5$).   As there are not enough vertices to form a blue/green $K_k$, it follows that $gr^3_2(K_4, K_5, K_k)\ge k$.  To prove the reverse inequality, consider a Gallai $3$-coloring of $K_k$ that avoids a red/blue $K_4$, a red/green $K_5$, and a blue/green $K_k$.   It cannot lack a red edge, or it would form a blue/green $K_k$.  So, let $x$ be a vertex that is incident with a red edge.  Let $N_r(x)$, $N_b(x)$, and $N_g(x)$ be the red, blue, and green neighborhoods of $x$, respectively.  If either of $N_r(x)$ or $N_b(x)$ contains a red or blue edge, then that edge, along with $x$ and a single vertex from the other neighborhood, forms a red/blue $K_4$.  So, both $N_r(x)$ and $N_b(x)$ induce complete green subgraphs.  

If both of $N_r(x)$ and $N_g(x)$ contain red or green edges, then those edges, along with $x$, form a red/green $K_5$.  So, either $N_g(x)$ induces a blue complete subgraph or $|N_r(x)|=1$.  If $N_g(x)$ only contains blue edges, then $|N_g(x)|\le 3$ (otherwise a blue $K_4$ exists).  Then $$|N_r(x)\cup N_b(x)|\ge k-4\ge 11.$$  By the pigeonhole principle, one of $N_r(x)$ and $N_b(x)$ contains a green $K_6$.  If $|N_r(x)|=1$, then $|N_b(x)|\le 4$ (since a green $K_5$ is avoided) implies that $$|N_g(x)|\ge k-6\ge 9.$$  Since, $r(K_3, K_4)=9$ \cite{GG}, it follows that the subgraph induced by $N_g(x)$ contains a red/green $K_3$ (which, along with $x$ and the vertex in $N_r(x)$, forms a red/green $K_5$) or a blue $K_4$.  
It follows that $gr^3_2(K_4, K_5, K_k)\le k$.
\end{proof}

\section{Conclusion}\label{con}

A natural extension of the work contained in this paper is to consider chromatic Gallai-Ramsey numbers of the form $gr^t_{t-1}(G_1, G_2, \dots , G_t)$, when $t\ge 4$.  Also, considering graphs other than complete graphs for the arguments leads to many new directions.  Also, the star-critical analogues of these numbers may also be considered.  Such numbers determine the number of edges that must be added to a critical coloring from a vertex so that the Ramsey property holds.  When the arguments are all complete graphs, the weakened star-critical chromatic Gallai-Ramsey number is $1$ less than the corresponding weakened chromatic Gallai-Ramsey number (cf., Theorem 3.8 of \cite{B2}).  So, the interesting cases will be those in which the arguments are not all complete graphs.

\

\noindent {\bf Acknowledgements.}  The first-listed author was funded by a Scholarly Development Assignment Program award from Western Carolina University.

\bibliographystyle{amsplain}



\end{document}